\documentclass[11pt,reqno]{amsart}
\usepackage[letterpaper,margin=1in]{geometry}
\usepackage[T1]{fontenc}
\usepackage{lmodern}
\usepackage{microtype}
\usepackage{amsmath,amssymb,mathtools}
\usepackage[colorlinks=true,linkcolor=blue,citecolor=blue,urlcolor=blue]{hyperref}
\hypersetup{pdftitle={Duflo-Serganova functor for chiral differential operators on classical invariant rings},
 pdfauthor={Tomoyuki Arakawa, Xuanzhong Dai, Bailin Song}}
\numberwithin{equation}{section}
\newtheorem{theorem}{Theorem}[section]

\newtheorem{corollary}[theorem]{Corollary}
\theoremstyle{remark}

\newcommand{\g}{\mathfrak g}
\newcommand{\osp}{\mathfrak{osp}}
\newcommand{\so}{\mathfrak{so}}
\newcommand{\spalg}{\mathfrak{sp}}
\newcommand{\DS}{\operatorname{DS}}
\newcommand{\sch}{\operatorname{sch}}
\newcommand{\ch}{\operatorname{ch}}
\newcommand{\tr}{\operatorname{tr}}
\newcommand{\rk}{\operatorname{rk}}
\newcommand{\A}{\mathcal A}
\newcommand{\V}{\mathbf V}
\newcommand{\SF}{\operatorname{SF}}
\newcommand{\Tot}{\operatorname{Tot}}

\title[DS functor for chiral differential operators]
{Duflo--Serganova functor for chiral differential\\
operators on classical invariant rings}
\author{Tomoyuki Arakawa}
\address[Tomoyuki Arakawa]{Okinawa Institute of Science and Technology,
\newline 1919-1 Tancha, Onna-son, Kunigami-gun, Okinawa 904-0495, Japan}
\email{tomoyuki.arakawa@oist.jp}
\author{Xuanzhong Dai}
\address[Xuanzhong Dai]{School of Mathematics, Nanjing University,
Nanjing 210093, China}
\email{xzdai@nju.edu.cn}
\author{Bailin Song}
\address[Bailin Song]{School of Mathematical Sciences,
University of Science and Technology of China,
Hefei, Anhui 230026, P. R. China}
\email{bailinso@ustc.edu.cn}
\date{}
\begin{document}
\begin{abstract}
We study the Duflo--Serganova functor for vertex algebras of chiral differential operators on classical invariant rings, introduced in \cite{ADS}. Our main result is a general cancellation theorem showing that a common bosonic and fermionic summand can be removed by Duflo--Serganova reduction. We apply this theorem to obtain Duflo--Serganova reduction isomorphisms for two families of simple orthosymplectic affine vertex superalgebras at levels  $-2n$ and $n$, which are not covered by the existing results of Gorelik and Serganova. As a consequence, we obtain supercharacter identities, including the conjecture of Adamović and Milas as a special case.
\end{abstract}
\maketitle

\section{Introduction and main results}
The Duflo--Serganova functor was introduced by Duflo and Serganova
in 2005 in their study of associated varieties of representations
of Lie superalgebras \cite{DS}. For a Lie superalgebra
$\mathfrak a$ and an odd element $x$ with $[x,x]=0$, it sends
an $\mathfrak a$-module $M$ to
\[
\DS_x(M)=\ker_M x/\operatorname{im}_M x,
\]
a module over
$\mathfrak a_x=\ker\operatorname{ad}x/\operatorname{im}\operatorname{ad}x$.
For graded modules with finite-dimensional homogeneous components and a degree-zero $x$, 
taking cohomology preserves the superdimension of each homogeneous component, and hence the homogeneous supercharacter is unchanged.

Gorelik and Serganova extended this construction to affine Lie
superalgebras and vertex superalgebras \cite{GS}.
Its compatibility with vertex algebra products follows from the
fact that every zero mode is a derivation. In particular, for any
conformal vertex superalgebra $V$ and odd weight-one element
$j\in V$ satisfying $j_{(0)}j=0$, the operator $j_{(0)}$ is
square-zero and one can define
\[
\DS_j(V)=\ker_V j_{(0)}/\operatorname{im}_V j_{(0)}.
\]
This is a vertex superalgebra when the vacuum class is nonzero,
and the same construction gives a functor from $V$-modules to
$\DS_j(V)$-modules. If $j_{(0)}$ annihilates the conformal vector,
as for a primary weight-one current, the conformal structure also
descends. Since $j_{(0)}$ preserves conformal weight, this reduction
preserves the homogeneous supercharacter when the weight spaces
are finite-dimensional.

In \cite{ADS}, the vertex algebras $\V_G(W_0,W_1)$ were
introduced as chiral differential operators on invariant rings by
relative BRST reduction of the $\beta\gamma$--$bc$ system
associated with two representations $W_0,W_1$ of a connected
reductive group $G$. Their geometric meaning is as follows.
The bosonic representation determines the invariant ring
$\mathcal O(W_0)^G$ and affine quotient
$\widetilde X=\operatorname{Spec}\mathcal O(W_0)^G$.
On a free closed-orbit locus $U\subset W_0$, put
$X=U/G$ and $E=(U\times W_1)/G$. The BRST sheaf on
$\widetilde X$ restricts to the sheaf of chiral differential
operators $\mathcal D_{\wedge E}^{\mathrm{ch}}$ 
\cite{GMS03}
on $X$;
under the geometric hypotheses of \cite{ADS}, its global sections
identify with $\V_G(W_0,W_1)$. The representation $W_1$
provides the fermionic fields needed for anomaly cancellation.

We study the DS functor for these chiral differential operator algebras. Our main result is a general cancellation theorem showing that a common bosonic and fermionic summand can be removed by DS reduction. This provides a mechanism for relating chiral differential operator algebras associated with different pairs of representations and, through their affine realizations, for producing DS reduction isomorphisms for affine vertex superalgebras.

Write
$\g=\operatorname{Lie}G$ and let $\kappa_\g$ be its Killing form.
The anomaly and trace conditions are
\begin{equation}\label{eq:anomaly}
\tr_{W_0}(\xi\eta)-\tr_{W_1}(\xi\eta)=\kappa_\g(\xi,\eta),
\qquad
\tr_{W_0}\xi=\tr_{W_1}\xi
\quad(\xi,\eta\in\g).
\end{equation}

\begin{theorem}[Cancellation of a common summand]\label{thm:cancellation}
Let $G$ be a connected complex reductive algebraic group and let
$W_0,W_1$ be finite-dimensional rational $G$-modules satisfying
\eqref{eq:anomaly}. Suppose that a nonempty $G$-stable open
subset $U\subset W_0$ has a free $G$-action and all its orbits
are closed in $W_0$.
Let $W_2$ be any finite-dimensional rational $G$-module and set
$W_0'=W_0\oplus W_2$, $W_1'=W_1\oplus W_2$.
For dual bases $\{t_i\},\{t^i\}$ of $W_2,W_2^*$, the odd current
\[
j=\sum_i:\beta^{t_i}c^{t^i}:
\]
descends to $\V_G(W_0',W_1')$ and defines a square-zero differential
$\delta=j_{(0)}$. There is an isomorphism
\begin{equation}\label{eq:cancellation}
H_\delta^p\bigl(\V_G(W_0',W_1')\bigr)
\cong
\begin{cases}
\V_G(W_0,W_1),&p=0,\\
0,&p\ne0.
\end{cases}
\end{equation}
The degree assigns $1,-1$ to the additional $c,b$ fields and
zero to the original matter fields. The isomorphism preserves
canonical conformal vectors and is equivariant for automorphisms
of $G$ and the three representations preserving the displayed
direct sums and the common identification of $W_2$.
No codimension or affinization hypothesis is required.
\end{theorem}

The construction of $\V_G(W_0,W_1)$ and its canonical conformal
vector is recalled in Section~2, and the theorem is proved in
Section~3. The additional degree in the theorem refines the DS
cohomology; forgetting this degree gives
$\DS_j(\V_G(W_0',W_1'))\cong\V_G(W_0,W_1)$.

\smallskip

Theorem 1.1 is the basic mechanism underlying our applications. 
By combining it with affine realizations of the vertex algebras  $\V_G(W_0,W_1)$, 
we obtain DS reduction isomorphisms for two families of simple orthosymplectic affine vertex superalgebras. 
These families lie outside the scope of the existing DS reduction results of Gorelik and Serganova, as explained below.


Let $L_k(\mathfrak a)$ denote the simple affine vertex superalgebra
of $\mathfrak a$ at level $k$, with its Sugawara conformal vector.
For orthosymplectic algebras, we normalize the invariant form so
that long roots of the orthogonal factor have squared length two.
For $x\in\mathfrak a_{\bar1}$ with $[x,x]=0$, we identify $x$
with its weight-one current and write $\DS_x$ for the cohomology
of $x_{(0)}$. For $\mathfrak a=\osp_{a|2b}$ and $\rk x=r$,
\[
\mathfrak a_x\cong\osp_{a-2r|2b-2r}.
\]
Here rank is half the rank of $x$ on the natural module.

\begin{corollary}\label{thm:sp}
Let $n\geq1$ and $M,N\geq0$ satisfy $M-N=2n+2$.
For a self-commuting odd element $x\in\osp_{2M|2N}$ of rank $N$,
there is an isomorphism of conformal vertex algebras
\begin{equation}\label{eq:sp}
\DS_x\bigl(L_{-2n}(\osp_{2M|2N})\bigr)
\cong L_{-2n}(\so_{4n+4}).
\end{equation}
\end{corollary}

The second application concerns a family obtained as an orthogonal orbifold.
\begin{corollary}\label{thm:so}
Let $n\geq5$, $N\geq3$, and $M-N=n-2$.
For a self-commuting odd element $x\in\osp_{2N|2M}$ of rank
$0\leq r\leq N-3$, there is an isomorphism of conformal vertex
superalgebras
\begin{equation}\label{eq:so}
\DS_x\bigl(L_n(\osp_{2N|2M})\bigr)
\cong L_n(\osp_{2(N-r)|2(M-r)}).
\end{equation}
In particular, rank-$(N-3)$ reduction gives
$\DS_x(L_n(\osp_{2N|2M}))\cong L_n(\osp_{6|2n+2})$.
\end{corollary}

The DS reduction isomorphisms in Corollaries 1.2 and 1.3 are not covered by the existing results of Gorelik and Serganova.
Gorelik and Serganova established DS isomorphisms for simple affine
vacuum modules at integrable levels
\cite[Theorem 2.2 and Corollary 3.4.2]{GS}, and at principal
admissible levels under the hypotheses of
\cite[Theorem 4.4.2]{GS}. 
The integrable-level results concern maximal-rank reduction, while in the principal admissible case the residual finite-dimensional algebra is assumed to be simple and of zero defect.
Their integrability convention requires integrability over the
finite-dimensional even part and local nilpotence of the root
operators for affine roots of positive squared length. Thus it is
a partial integrability condition, with respect to a prescribed
affine even factor $\widehat{\mathfrak a^\#}$, rather than to an
arbitrary even factor; see \cite[p.~4]{GS} and
\cite[Section 5.2]{GSInt}.

In Corollary~\ref{thm:sp}, $\mathfrak a^\#=\so_{2M}$ has
induced level $-2n$, whereas the integrable affine factor is
$\widehat{\spalg}_{2N}$ at level $n$ when $N>0$.
In Corollary~\ref{thm:so}, $\mathfrak a^\#=\spalg_{2M}$ has
induced level $-n/2$, whereas $\widehat{\so}_{2N}$ acts
integrably at level $n$. The finite-dimensional even actions
are integrable in both cases, but the negative levels on
$\widehat{\mathfrak a^\#}$ exclude integrability in the
sense of \cite{GS}.

The principal admissible vacuum theorem does not apply either.
With the invariant form normalized on $\mathfrak a^\#$ as in
\cite{GS}, the level and dual Coxeter number are
$(k,h^\vee)=(-2n,4n+2)$ in Corollary~\ref{thm:sp} and
$(-n/2,n-1)$ in Corollary~\ref{thm:so}.
In each case the reduced numerator of $k+h^\vee$ is smaller
than $h^\vee$, violating the principal admissibility condition
in \cite[Section 4.3]{GS}. Moreover, for the nontrivial
reductions in Corollary~\ref{thm:so}, $r<N$ and the residual
algebra has positive defect $N-r\geq3$; thus the maximal-rank
and zero-defect hypotheses also fail.

The proof of the general cancellation theorem proceeds by reversing the order of the odd-current and relative BRST cohomologies. Its equivariance also yields the orthogonal orbifold statement.

Finally, the DS reduction isomorphisms yield supercharacter identities for these orthosymplectic affine vertex superalgebras.
Since DS reduction preserves homogeneous supercharacters, 
Corollaries \ref{thm:sp} and \ref{thm:so} imply the following identities. 
In particular, 
specializing the first family to $n=1$
recovers the Adamović--Milas conjecture.

For a conformal vertex superalgebra with finite-dimensional weight
spaces, write
\[
\sch V(q)=\operatorname{str}_V q^{L_0-c/24},
\qquad
\ch V(q)=\tr_V q^{L_0-c/24}.
\]
\begin{corollary}\label{cor:characters}
Under the respective hypotheses of Corollaries~\ref{thm:sp}
and~\ref{thm:so},
\begin{align}
\sch L_{-2n}(\osp_{2M|2N})(q)
 &=\ch L_{-2n}(\so_{4n+4})(q),\label{eq:char-sp}\\
\sch L_n(\osp_{2N|2M})(q)
 &=\sch L_n(\osp_{6|2n+2})(q).\label{eq:char-so}
\end{align}
In particular, for every $N\geq0$,
\begin{equation}\label{eq:AM}
\sch L_{-2}(\osp_{2N+8|2N})(q)
=\ch L_{-2}(\so_8)(q)
=\frac{q\,\frac{d}{dq}E_4(q)}{240\eta(q)^{10}},
\end{equation}
where $E_4(q)=1+240\sum_{j\geq1}\sigma_3(j)q^j$ and
$\eta(q)=q^{1/24}\prod_{j\geq1}(1-q^j)$.
\end{corollary}
Equation~\eqref{eq:AM} is the conjecture of Adamovi\'c and Milas
\cite[Conjecture 8.1]{AM}. The last expression is the known character
formula for $L_{-2}(\so_8)$ \cite{AK}.

\medskip
\noindent
\textbf{Acknowledgements}.
Tomoyuki Arakawa is partially supported by JSPS KAKENHI Grant Numbers
21H04993, 25K21659 and 26H019. 
Part of this paper was written during his stay at SLMath 
for the program
``Representation Theory Under the Influence of Quantum Field Theory".
He thanks the institute for its excellent working environment. Bailin Song is supported by NSFC No. 12671042.
\medskip

\noindent
\textbf{Declaration of AI use.}
During the preparation of this work, the authors used OpenAI’s ChatGPT (GPT-6) to assist with mathematical discussions, the development and exposition of proofs, literature searches, and the drafting and editing of the manuscript. The authors independently verified all mathematical arguments and references, revised the AI-assisted material, and take full responsibility for the content of the paper.

\section{Relative BRST reduction}
We use the hypotheses and notation of Theorem~\ref{thm:cancellation}.
The conditions \eqref{eq:anomaly} and the assumption on $U$ are
the level and trace conditions and assumptions (A1)--(A2) of
\cite{ADS}.

Set $\A(W_0,W_1)=\mathcal S(W_0)\otimes\mathcal E(W_1)$.
For dual bases $\{v_i\},\{v^i\}$ of $W_0,W_0^*$ and
$\{w_s\},\{w^s\}$ of $W_1,W_1^*$, the even fields
$\beta^{v_i},\gamma^{v^i}$ and odd fields $b^{w_s},c^{w^s}$ have
singular OPEs
\[
\beta^{v_i}(z)\gamma^{v^j}(w)\sim\frac{\delta_{ij}}{z-w},
\qquad
b^{w_s}(z)c^{w^t}(w)\sim\frac{\delta_{st}}{z-w}.
\]
The gauge currents
\[
J^\xi=-\sum_i:\beta^{\xi v_i}\gamma^{v^i}:
+\sum_s:b^{\xi w_s}c^{w^s}:
\]
have level $-\kappa_\g$ by \eqref{eq:anomaly}.
The relative BRST complex and its degree-zero cohomology are
\begin{equation}\label{eq:relative}
C^q(W_0,W_1)
=\bigl(\A(W_0,W_1)\otimes\SF^q(\g)\bigr)^{\g},
\qquad
\V_G(W_0,W_1)=H^0(C(W_0,W_1),D).
\end{equation}
Here $\SF(\g)$ is generated by the antighosts $B^u$ and the
fields $\partial C^\alpha$, of ghost degrees $-1,1$.
For dual bases $\{u_a\}\subset\g$, $\{u^a\}\subset\g^*$,
the differential is $D=Q_{(0)}$, where
\[
Q=\sum_a:J^{u_a}C^{u^a}:
-\frac12\sum_{a,b}:C^{u^a}C^{u^b}B^{[u_a,u_b]}:.
\]
The relative condition removes the ghost zero modes.

The canonical conformal vector is
\begin{equation}\label{eq:canonical}
L_{\mathrm{can}}=
\sum_i L_{\beta^{v_i},\gamma^{v^i}}
+\sum_s L_{b^{w_s},c^{w^s}}+L_{\mathrm{gh}},
\end{equation}
where
\[
L_{\beta,\gamma}=:\beta\partial\gamma:
-\frac12\partial(:\beta\gamma:),\quad
L_{b,c}=-:b\partial c:+\frac12\partial(:bc:),\quad
L_{\mathrm{gh}}=-\sum_a:B^{u_a}\partial C^{u^a}:.
\]
All matter generators have weight $1/2$, and the generators
$B,\partial C$ of $\SF(\g)$ have weight one.
Thus every canonical weight space of the relative complex is
finite-dimensional. By the BRST vanishing theorem of \cite{ADS},
\begin{equation}\label{eq:vanishing}
H^q(C(W_0,W_1),D)=0\qquad(q\ne0).
\end{equation}

\section{Proof of the cancellation theorem}
First, the hypotheses giving \eqref{eq:vanishing} hold for
$(W_0',W_1')$ as well. The trace contributions of $W_2$ cancel in
\eqref{eq:anomaly}. On $U'=U\times W_2$ the action is free.
For $(u,w)\in U'$, its orbit is the graph of a regular map from
the closed orbit $G\cdot u\cong G$ to $W_2$, and is therefore
closed in $W_0'$. Hence, writing $C'=C(W_0',W_1')$,
\begin{equation}\label{eq:vanishing-prime}
H_D^q(C')\cong\delta_{q,0}\V_G(W_0',W_1').
\end{equation}

The current $j$ is independent of the chosen dual bases and has
regular OPEs with itself and with all gauge currents.
It is therefore a relative BRST cocycle and descends to
$\V_G(W_0',W_1')$.
On the additional matter fields its zero mode acts by
\begin{equation}\label{eq:pairs}
\delta\gamma^{t^i}=c^{t^i},\qquad
\delta b^{t_i}=\beta^{t_i},\qquad
\delta\beta^{t_i}=\delta c^{t^i}=0,
\end{equation}
and it annihilates the original matter.
In particular, $\delta^2=0$. The additional gauge current is exact:
\[
J_{\mathrm{add}}^\xi
=-\delta\left(\sum_i:b^{\xi t_i}\gamma^{t^i}:\right).
\]
Extend $\delta$ by zero on the ghosts. It preserves the relative
complex and satisfies $D\delta+\delta D=0$.
Moreover, $D$ preserves the additional matter degree in
Theorem~\ref{thm:cancellation}. We obtain a double complex with
$\delta,D$ of bidegrees $(1,0),(0,1)$.

As differential vertex superalgebras,
\[
\A(W_0',W_1')=
\A(W_0,W_1)\otimes
\bigl(\mathcal S(W_2)\otimes\mathcal E(W_2)\bigr),
\]
with zero differential on the first factor.
In the second factor all creation modes form the contractible
pairs in \eqref{eq:pairs}, together with their derivatives.
More explicitly, on their PBW symmetric superalgebra, extend
$c\mapsto\gamma$, $\beta\mapsto b$ as an odd superderivation.
Its anticommutator with $\delta$ is total PBW degree, so every
positive-degree vector is contractible. Thus
\begin{equation}\label{eq:matter-cohom}
H_\delta^p\bigl(\A(W_0',W_1')\bigr)
\cong\delta_{p,0}\A(W_0,W_1).
\end{equation}
Since the representations are rational and $G$ is reductive,
taking $\g$-invariants is exact on the modules at hand. Therefore
\[
H_\delta^p(C'^\bullet)
\cong\delta_{p,0}C^\bullet(W_0,W_1).
\]
The induced differential is the original BRST differential:
the additional gauge currents have zero class, and the ghost
term is unchanged.

Both differentials preserve canonical conformal weight.
At each conformal weight, the double complex is finite-dimensional; 
hence both cohomological degrees are bounded and both spectral sequences converge
to $H^\bullet(\Tot C',D+\delta)$.
Taking $D$-cohomology first
and using
\eqref{eq:vanishing-prime},
we identify this with
$H_\delta^\bullet(\V_G(W_0',W_1'))$.
Taking $\delta$-cohomology first and then using
\eqref{eq:vanishing} identifies it with $\V_G(W_0,W_1)$,
concentrated in degree zero. This proves \eqref{eq:cancellation}.
The spectral sequences respect all vertex algebra products.
BRST cocycles formed from the original matter and no ghosts
are unchanged by the identifications; in particular, the
isomorphism intertwines any residual affine currents of this form.

For the contribution of the additional matter to
\eqref{eq:canonical}, one has
\[
L_{\mathrm{can}}^{W_2}
=\delta\left(
\sum_i:b^{t_i}\partial\gamma^{t^i}:
-\frac12\partial\sum_i:b^{t_i}\gamma^{t^i}:
\right).
\]
Consequently the canonical conformal vector maps to that of
the original complex. Finally, the pairing defining $j$ and
the entire construction are equivariant under the compatible
automorphisms in the statement; on ghosts they act through
the induced automorphisms of $\g$ and $\g^*$.
This proves Theorem~\ref{thm:cancellation}.

In particular, if a finite group $\Gamma$ of such automorphisms
acts on the data, exactness of invariants gives
\begin{equation}\label{eq:equivariance}
H_\delta^\bullet\bigl(\V_G(W_0',W_1')^\Gamma\bigr)
\cong\V_G(W_0,W_1)^\Gamma.
\end{equation}

\section{DS reductions and supercharacters}
We recall the affine identifications from \cite{ADS}.
For $G=Sp_{2n}$ or $SO_n$ let $V$ be the standard representation
and abbreviate
$\V^G_{M,N}=\V_G(V^{\oplus M},V^{\oplus N})$.
The required geometric ranges and identifications are
\begin{equation}\label{eq:ranges}
\begin{array}{c|c}
G&\text{parameters}\\ \hline
Sp_{2n}&n\geq1,\quad M-N=2n+2,\quad N\geq0,\\
SO_n&n\geq5,\quad M-N=n-2,\quad N\geq3,
\end{array}
\end{equation}
and
\begin{align}
\V^{Sp_{2n}}_{M,N}&\cong L_{-2n}(\osp_{2M|2N}),\label{eq:id-sp}\\
(\V^{SO_n}_{M,N})^{\mathbb Z_2}
&\cong L_n(\osp_{2N|2M}),\qquad
\mathbb Z_2=O_n/SO_n.\label{eq:id-so}
\end{align}
Here the free locus consists of full-rank matrices.
The identifications preserve canonical and Sugawara conformal
vectors; their central charges are respectively
$-8n^2-6n$ and $-n(2n-3)$.
For \eqref{eq:id-so}, $O_n$ acts on matter through its standard
representation and on ghosts through its adjoint representation
and its dual. On BRST cohomology this action factors through
$O_n/SO_n$. 
We use the geometric hypotheses and affine identifications established in \cite[Section 6]{ADS}, 
but not the supercharacter formulas obtained there.

\begin{proof}[Proof of Corollary~\ref{thm:sp}]
Apply Theorem~\ref{thm:cancellation} to $G=Sp_{2n}$ with
$W_0=V^{\oplus(2n+2)}$, $W_1=0$, and $W_2=V^{\oplus N}$.
Numbering the additional columns last, the odd current is
\begin{equation}\label{eq:delta}
j_N=\sum_{s=1}^N\sum_{a=1}^{2n}
:\beta^{a,2n+2+s}c^{as}:.
\end{equation}
Under \eqref{eq:id-sp}, it is the current of a rank-$N$
self-commuting odd flavor element: its action on the natural
flavor module has $N$ pairs as in \eqref{eq:pairs} and
cohomology of dimension $(2M-2N\mid0)$.
The residual flavor algebra is $\so_{2(M-N)}$, and the
retained currents give the corresponding affine action.
The theorem and \eqref{eq:id-sp} give
\[
\DS_x(\V^{Sp_{2n}}_{M,N})
\cong\V^{Sp_{2n}}_{2n+2,0}
\cong L_{-2n}(\so_{4n+4})
\]
as conformal vertex algebras. Conjugacy of rank-$N$
self-commuting odd elements under the even flavor group gives
the result for the stated $x$; see \cite{DS}.
\end{proof}

\begin{proof}[Proof of Corollary~\ref{thm:so}]
For $G=SO_n$, take $W_0=V^{\oplus(M-r)}$,
$W_1=V^{\oplus(N-r)}$, and $W_2=V^{\oplus r}$.
The original pair is in \eqref{eq:ranges} because $N-r\geq3$.
The current
\[
j_r=\sum_{s=1}^r\sum_{a=1}^n
:\beta^{a,M-r+s}c^{a,N-r+s}:
\]
is $O_n$-invariant and represents a rank-$r$ self-commuting
odd element of $\osp_{2N|2M}$. The orthogonal realization
uses the parity reversal of the flavor space of free fields.
The residual flavor algebra is
$\osp_{2(N-r)|2(M-r)}$.
Theorem~\ref{thm:cancellation} and \eqref{eq:equivariance} give
\[
\DS_x\bigl((\V^{SO_n}_{M,N})^{\mathbb Z_2}\bigr)
\cong(\V^{SO_n}_{M-r,N-r})^{\mathbb Z_2}.
\]
Equation~\eqref{eq:id-so} proves \eqref{eq:so}, preserving the
Sugawara vectors. Conjugacy gives the statement for any $x$
of rank $r$.
\end{proof}

\begin{proof}[Proof of Corollary~\ref{cor:characters}]
The odd differential $\delta$ commutes with the canonical grading,
and every weight space is finite-dimensional. Its cohomology
therefore has the same superdimension as the original space in
each weight. The conformal isomorphisms in
Corollaries~\ref{thm:sp} and~\ref{thm:so} preserve central charge,
so the normalized homogeneous supercharacters agree.
Taking $r=N-3$ gives \eqref{eq:char-so}, and taking $n=1$ in
\eqref{eq:char-sp}, together with the $D_4$ character formula
of \cite{AK}, gives \eqref{eq:AM}.
\end{proof}


\begin{thebibliography}{99}
\bibitem{AM}
D. Adamovi\'c and A. Milas,
On some vertex algebras related to $V_{-1}(\mathfrak{sl}(n))$
and their characters,
Transform. Groups \textbf{26} (2021), no. 1, 1--30.

\bibitem{ADS}
T. Arakawa, X. Dai, and B. Song,
Faithful Localization of Vertex Algebra Modules,
preprint.

\bibitem{AK}
T. Arakawa and K. Kawasetsu,
Quasi-lisse vertex algebras and modular linear differential equations,
in \emph{Lie Groups, Geometry, and Representation Theory:
A Tribute to the Life and Work of Bertram Kostant},
Progr. Math. 326, Birkh\"auser, Cham, 2018, 41--57.

\bibitem{GMS03} V. Gorbounov, F. Malikov, and V. Schechtman, Gerbes of chiral differential operators. III, in The Orbit Method in Geometry and Physics, Progr. Math. 213, Birkh\"auser, Boston, 2003, 73--100.

\bibitem{DS}
M. Duflo and V. Serganova,
On associated variety for Lie superalgebras,
\href{https://arxiv.org/abs/math/0507198}{arXiv:math/0507198} (2005).

\bibitem{GS}
M. Gorelik and V. Serganova,
On DS functor for affine Lie superalgebras,
RIMS K\^oky\^uroku \textbf{2075} (2018), 127--146;
\href{https://arxiv.org/abs/1711.10149}{arXiv:1711.10149}.

\bibitem{GSInt}
M. Gorelik and V. Serganova,
Integrable modules over affine Lie superalgebras
$\mathfrak{sl}(1|n)^{(1)}$,
Commun. Math. Phys. \textbf{364} (2018), 635--654.
\end{thebibliography}
\end{document}